\documentclass[12pt]{article}
\usepackage{amsmath,amssymb,amsfonts,amsthm,mdframed,graphicx}
\usepackage{tikz}
\usepackage[colorlinks,
linkcolor=blue,
anchorcolor=blue,
citecolor=blue
]{hyperref}

\DeclareMathOperator{\red}{red}

\newtheorem{thm}{Theorem}[section]
\newtheorem{deff}[thm]{Definition}
\newtheorem{lem}[thm]{Lemma}

\definecolor{blue}{rgb}{0,0,1}
{}

\begin{document}

\begin{center}
{\large \bf  On a Greedy Algorithm to Construct Universal \\[5pt]
	  Cycles for Permutations}
\end{center}

\begin{center}
Alice L.L. Gao$^{1}$,
Sergey Kitaev$^{2}$, 
Wolfgang Steiner$^{3}$
and Philip B. Zhang$^{4}$\\[6pt]

$^{1}$Department of Applied Mathematics\\
Northwestern Polytechnical University,
Xi’an, Shaanxi 710072, P.R. China\\[6pt]

$^{2}$Department of Computer and Information Sciences \\
University of Strathclyde, 26 Richmond Street, Glasgow G1 1XH, UK\\[6pt]

$^{3}$IRIF, CNRS UMR 8243 \\
Universit\'e Paris Diderot -- Paris 7, 75205 Paris Cedex 13, Paris, France\\[6pt] 

$^{4}$College of Mathematical Science \\
Tianjin Normal University, Tianjin  300387, P. R. China\\[6pt]

Email: $^{1}${\tt llgao@nwpu.edu.cn},
	   $^{2}${\tt sergey.kitaev@cis.strath.ac.uk},
           $^{3}${\tt steiner@irif.fr},
           $^{4}${\tt zhang@tjnu.edu.cn}
\end{center}

\noindent\textbf{Abstract.}
A universal cycle for permutations of length $n$ is a cyclic word or permutation, any factor of which is order-isomorphic to exactly one permutation of length $n$, and containing all permutations of length $n$ as factors. It is well known that universal cycles for permutations of length $n$ exist. However, all known ways to construct such cycles are rather complicated. For example, in the original paper establishing the existence of the universal cycles, constructing such a cycle involves finding an Eulerian cycle in a certain graph and then dealing with partially ordered sets.

In this paper, we offer a simple way to generate a universal cycle for permutations of length $n$, which is based on applying a greedy algorithm to a permutation of length $n-1$. We prove that this approach gives a unique universal cycle $\Pi_n$ for permutations, and we study properties of $\Pi_n$. \\



\section{Introduction}\label{intro}

A {\em universal cycle}, or {\em u-cycle}, for a class of combinatorial objects is a cyclic word containing each object, encoded by a word, exactly once as a cyclic interval (called a {\em factor}). Universal cycles were introduced by  Chung, Diaconis and Graham in~\cite{CDG1992}. However, the origin of the notion of a u-cycle is in the study of the celebrated de Bruijn cycles, which are cyclic words containing each word of length $n$ over some alphabet, for a given $n$, exactly once (see~\cite{deBruijn}). 

There are many ways to construct a de Bruijn cycle. A ``classical'' way to construct a de Bruijn cycle is via the notion of a {\em de Bruijn graph of order}~$n$, which is a directed graph representing overlaps between words of length~$n$. It can be easily shown that any de Bruijn graph is {\em balanced} and {\em strongly connected}, and thus it contains an {\em Eulerian cycle}. Moreover, one can show that a de Bruijn graph of order~$n$ is the {\em line graph} of the de Bruijn graph of order $n-1$, and thus any de Bruijn graph contains a {\em Hamiltonian cycle}, which can easily be  translated to a de Bruijn cycle.

\subsection{Martin's algorithm to generate de Bruijn sequences}
An alternative construction of a de Bruijn sequence involves concatenating together, in lexicographic order, all the {\em Lyndon words} whose length divides $n$. An interesting, and most relevant to our paper fact is that this sequence was first generated by Martin~\cite{Martin1934} in 1934 using the following greedy algorithm. 

\medskip

\begin{mdframed}
\centerline{\bf Martin's greedy algorithm to generate a de Bruijn sequence}

\medskip
Start with the word $(k-1)^{n-1}$ and then repeatedly apply the following rule: Append the {\em smallest} letter in $\{0,1,\ldots,k-1\}$ so that factors of length $n$ in the resulting word are distinct. Once no more extension is possible, remove the $n-1$ rightmost letters.
\end{mdframed}

For example, for $k=3$ and $n=2$, the steps of the algorithm are: $2\rightarrow 20 \rightarrow 200 \rightarrow 2001 \rightarrow 20010 \rightarrow 200102 \rightarrow 2001021 \rightarrow 20010211 \rightarrow 200102112\rightarrow 20010211$.

\subsection{Universal cycles for permutations}
In this paper, permutations of length $n$ are called {\em $n$-permutations}.

Universal cycles for permutations were one of the main objects considered in~\cite{CDG1992}. A~universal cycle for $n$-permutations is a cyclic word, each length $n$ factor of which is order-isomorphic to a unique $n$-permutation, and every $n$-permutation is order-isomorphic to some factor in the word. There is a long line of research on universal cycles for permutations, e.g.\  see \cite{AW2009,HRW2010,HRW2012,HI1996,I2006,J1993,Johnson2009,RW2010}. 

Even though the basic original idea to construct universal cycles for permutations is similar to that of constructing de Bruijn cycles via de Bruijn graphs, the situation with permutations is much more complicated. Indeed, just proving the existence of such a universal cycle uses a not so trivial trick of {\em clustering} the {\em graph of non-overlapping permutations}, an analogue of the de Bruijn graph in the case of permutations. The actual construction of universal cycles for permutations in~\cite{CDG1992} involved dealing with partially ordered sets and resulted in a nice conjecture on the minimal alphabet that was settled 17 years later by Johnson in~\cite{Johnson2009}. Namely, it turns out that it is possible to construct a universal cycle for $n$-permutations using just $n+1$ letters; using $n$ letters is impossible (unless using the {\em shorthand encoding of $n$-permutations by $(n-1)$-permutations, which is deleting the rightmost element in each $n$-permutation; see~\cite{HRW2012}}) as can be easily seen, e.g.\ in the case of $n=3$. 

We found the existent constructions of the universal cycles for permutations in the literature to be rather involved, and we asked ourselves whether the simple algorithm by Martin for de Bruijn cycles could be extended to the case of the permutations. We were pleased to discover that the answer to the question was positive with surprisingly small changes that were required to Martin's algorithm: instead of the initial word  $(k-1)^{n-1}$ we use the increasing permutation of length $n-1$, and instead of the smallest available letter on a given iteration of the process we use the smallest available extension of the respective permutation, to be defined in Section~\ref{gener-sec}. Interestingly, beginning with a monotone permutation, that is, the increasing or decreasing permutation of appropriate length, is a necessary condition in our greedy algorithm, as we show in Section~\ref{unique-sub}.

\subsection{Basic definitions}
For a permutation, or word, $\pi$, the {\em reduced form} of $\pi$, denoted $\red(\pi)$, is obtained by replacing the $i$-th smallest element in $\pi$ by $i$. For example, $\red(3275)=2143$. Any $i$ consecutive letters of a word or permutation $w$ form a {\em factor} of $w$. If $w$ is a cyclic word then a factor can begin at the end of $w$ and end at the beginning of $w$. If $u$ is a factor of $w$, we also say that $w$ {\em covers} $\red(u)$.

\begin{deff} A word $\Pi'_n$ is a {\em universal word}, or {\em u-word}, for $n$-permutations if $\Pi'_n$ covers {\em every} $n$-permutation {\em exactly once}.
\end{deff}

\begin{deff} A cyclic word $\Pi_n$ is a {\em universal cycle}, or {\em u-cycle}, for $n$-permutations if $\Pi_n$ covers {\em every} $n$-permutation {\em exactly once}.
\end{deff}

\subsection{Organization of the paper} This paper is organized as follows. In Section~\ref{gener-sec} we show our simple way to generate the words $\Pi'_n$ and $\Pi_n$, and in Sections~\ref{just-1} and~\ref{just-2} we justify that the constructions of $\Pi'_n$ and $\Pi_n$ work. Also, in Section~\ref{unique-sub} we show that our greedy algorithm to generate $\Pi'_n$ and $\Pi_n$ can only be applied to the increasing permutation of length $n-1$. Further, Section~\ref{prop-sec} discusses some properties of $\Pi'_n$ and $\Pi_n$, and Section~\ref{further-res-sec} gives several directions of future research. This includes a discussion on reducing the alphabets of $\Pi'_n$ and $\Pi_n$ in Section~\ref{reducing-sec} and a possibility of extending our approach to relevant object in the literature in Section~\ref{sec-shortened}.

\section{Simple generation of u-words and u-cycles for permutations}

\subsection{Generating $\Pi'_n$ and $\Pi_n$}\label{gener-sec}

Let $\pi = \pi_1\pi_2\cdots \pi_m$ be a permutation of $m$ distinct integers. 
Then the {\em $i$-th extension} of $\pi$ to the right, $1\leq i\leq m$, is the permutation  
$$c_b(\pi_1)c_b(\pi_2)\cdots c_b(\pi_m)b,$$ 
where $b$ is the $i$-th smallest element in $\{\pi_1,\pi_2,\ldots,\pi_m\}$, and
$$c_b(x)=
\begin{cases}
x & \mbox{if }x<b, \\
x+1 & \mbox{if }x\geq b.
\end{cases}
$$
The {\em $(m+1)$-st extension} is the permutation $\pi b$, where $b$ is the largest element in $\{\pi_1+1,\pi_2+1,\ldots,\pi_m+1\}$. We call the first extension the {\em smallest extension}, and the $(m+1)$-st extension the {\em largest extension} of~$w$.

The following simple algorithm produces recursively a universal word $\Pi'_n$ of length $n!+n-1$ for $n$-permutations.

\medskip

\begin{mdframed}
\centerline{\bf The greedy algorithm to construct $\Pi'_n$}

\medskip

Begin with the increasing permutation
$$\Pi'_{n,0}:=12\cdots(n-1).$$
Suppose that a permutation
$$\Pi'_{n,k}=a_{1}a_{2}\cdots a_{k+n-1}$$
has been constructed for $0 \leq k < n!$, and no two factors in $\Pi'_{n,k}$ of length $n$ are order-isomorphic. Let $i$ be minimal such that no factor of length~$n$ in $\Pi'_{n,k}$ is order-isomorphic to the $i$-th extension of $a_{k+1}a_{k+2}\cdots a_{k+n-1}$, and denote the last element of this extension by~$b$.
Then 
$$\Pi'_{n,k+1}:=c_b(a_1)c_b(a_2)\cdots c_b(a_{k+n-1})b.$$
For some $k^*$, no extension of $\Pi'_{n,k^*}$ will be possible without creating a factor order-isomorphic to a factor in $\Pi'_{n,k}$. The greedy algorithm then terminates and outputs $\Pi'_n:=\Pi'_{n,k^*}$.
\end{mdframed}

\medskip\noindent
For example, the steps of the algorithm for $n=3$ are as follows:
$$12  \rightarrow
 231 \rightarrow
 3421 \rightarrow
  45312 \rightarrow
 564132 \rightarrow  6751324 \rightarrow
 78613245 = \Pi'_3.$$
We note that for each $k$,  $\Pi'_{n,k}$ is a permutation of $\{1,2,\ldots,k+n-1\}$.

The following simple extension of the greedy algorithm turns the universal word for permutations $\Pi'_n$ into a universal cycle for permutations $\Pi_n$.

\medskip
\begin{mdframed}
\centerline{\bf Generating the u-cycle $\Pi_n$ from $\Pi'_n$}

\medskip

Remove the last $n-1$ elements in $\Pi'_n$ and take the reduced form of the resulting sequence to obtain $\Pi_n$.
\end{mdframed}

\medskip\noindent
For example, $\Pi_3$ is given by
$$\Pi'_3=78613245 \rightarrow 786132\rightarrow \red(786132)= 564132=\Pi_3.$$
For another example, $\Pi_4$ is given by
$$22\ 23\ 24\ 21\ 20\ 18\ 19\ 3\ 17\ 4\ 2\ 16\ 1\ 6\ 7\ 5\ 11\ 10\ 8\ 13\ 9\ 12\ 15\ 14.$$

\subsection{Justification of $\Pi'_n$ being a u-word for permutations}\label{just-1}
In the following, let $n$ be arbitrary but fixed.
For $\Pi'_{n,k} = a_{1}a_{2}\cdots a_{k+n-1}$, set
$$\sigma_{k} := \red(a_{k}a_{k+1}\cdots a_{k+n-1}), \qquad \sigma'_{k} := \red(a_{k+1}a_{k+2}\cdots a_{k+n-1}),$$
and 
$$J_{k} = |\{j\le k:\, \sigma'_{j} = \sigma'_{k}\}|,$$
i.e.\ $J_{k}$ is the number of occurrences of the $(n-1)$-permutation $\sigma'_{k}$ in~$\Pi'_{n,k}$.
By the definition of the greedy algorithm, all $i$-th extensions of~$\sigma'_{k}$ with $i < J_k$ occur in $\Pi'_{n,k}$, and the $J_k$-th extension of $\sigma'_{k}$ does not occur in~$\Pi'_{n,k}$.
Therefore, the greedy algorithm terminates at~$k$ if and only if $J_k = n+1$.
If $J_k \le n$, then $\sigma_{k+1}$ is the $J_k$-th extension of $\sigma'_{k}$ (and ends with~$J_k$).

\begin{lem}\label{lem2}
The greedy algorithm terminates at~$k$ if and only if $\sigma_{k} {\,=\,}12\cdots n$.
\end{lem}

\begin{proof}
If $\sigma_{k} = 12\cdots n$, then $\sigma'_{k-1} = \sigma'_{k} = 12\cdots(n-1)$ and $J_{k-1} = n$. Thus, $J_k = n+1$, and the greedy algorithm terminates at~$k$.

For the converse, assume that $\sigma_{k} \ne 12\cdots n$. 
By the preceding paragraph, this implies that $\sigma_{j} \ne 12\cdots n$ for all $j \le k$. 
Since each word $a_{j+1} \cdots a_{j+n-1}$, $1 \le j \le k$, is preceded by the letter $a_{j}$ and the permutation $\sigma_j$ occurs only at the position~$j$, we have $\sigma'_j = \sigma'_k$ for at most $n$ different indices $j \ge 1$.
If $\sigma'_{k} \ne 12\cdots (n-1)$, then we have $\sigma'_{0} \ne \sigma'_{k}$ and thus $J_k \le n$. 
If $\sigma'_{k} = 12\cdots (n-1)$, then we have $\sigma'_j = \sigma'_k$ for at most $n-1$ different indices $j \ge 1$ (because $\sigma_{j} = 12\cdots n$ is not possible), which also gives that $J_k \le n$. 
Therefore, the greedy algorithm does not terminate at~$k$ if $\sigma_{k} \ne 12\cdots n$.
\end{proof}

\begin{lem}\label{lem3}
$\Pi'_n$ covers all $n$-permutations.
\end{lem}

\begin{proof}
We shall prove this lemma by contradiction.
Suppose that a permutation $\pi_1 \pi_2 \cdots \pi_{n}$ is not covered by $\Pi'_n$. 
Then $\pi_2 \cdots \pi_{n}$ is covered at most $n-1$ times, hence $\red(\pi_2 \cdots \pi_{n})n$ is not covered. 
More generally, we have for $1 \le k \le n$ that, if $\red(\pi_k \cdots \pi_{n}) (n-k+2) \cdots n$ is not covered, then 
$$\red(\red(\pi_{k+1} \cdots \pi_{n}) (n-k+2) \cdots n)n = \red(\pi_{k+1} \cdots \pi_{n}) (n-k+1) \cdots (n-1)n$$
is not covered.
We obtain that $\Pi'_n$ does not cover the permutation $1 2\cdots n$, contradicting Lemma~\ref{lem2}.
\end{proof}

\begin{thm}\label{u-word-thm} 
$\Pi'_n$ is a u-word for $n$-permutations.
\end{thm}

\begin{proof}
By the nature of the greedy algorithm, $\Pi'_n$ cannot cover a permutation more than once. By Lemma~\ref{lem3}, $\Pi'_n$ covers all $n$-permutations.
\end{proof}

\subsection{Justification of \texorpdfstring{$\Pi_n$}{} being a u-cycle for permutations}\label{just-2}

\begin{thm}
 $\Pi_n$ is a u-cycle for $n$-permutations.
\end{thm}

\begin{proof}
By Theorem~\ref{u-word-thm}, it suffices to prove that 
$$\red(a_{k}a_{k+1} \cdots a_{k+n-1}) = \red(a_{k}\cdots a_{n!} a_1 \cdots a_{k+n-n!-1})$$ 
for all $n!-n+2 \le k \le n!$.

First, note that $\sigma_i$ ends with $1$ for all $i < n$. 
Hence, we have that
$$\Pi'_{n,n-1} = n(n+1)\cdots(2n-2) (n-1)(n-2)\cdots1$$
and thus
$$a_k < a_1 < a_2 < \cdots < a_{n-1}$$
{for all} $k \ge n$.

Next, we show that $\sigma_i$ ends with $n$ for all $i \ge n!-n+2$.
Suppose that this is not true for some $i \ge n!-n+2$. 
Since $\Pi'_n$ is a u-word, we must have $\sigma'_{i-1} = \sigma'_{j-1}$ for some $j > i$. 
It follows from $\sigma_{n!} = 12\cdots n$ that 
$$a_{n!} < a_{n!+1} < \cdots < a_{n!+n-1}.$$
In particular, $$a_{n!} < a_{n!+1} < \cdots < a_{j+n-2}.$$
Then, $\sigma'_{i-1} = \sigma'_{j-1}$ implies that 
$$a_{n!+i-j} < a_{n!+i-j+1} < \cdots < a_{i+n-2}< \cdots < a_{n!+n-1}.$$
Iterating this argument gives that $a_{j} < a_{j+1} < \cdots < a_{n!+n-1}$ and thus $a_{i} < a_{i+1} < \cdots < a_{n!+n-1}$, contradicting the assumption $\sigma_i$ does not end with~$n$.
Therefore, $\sigma_i$ ends with $n$ for all $i \ge n!-n+2$, which implies that 
$$a_k < a_{n!+1} < a_{n!+2} < \cdots < a_{n!+n-1}$$
for all $n!-n+2 \le k \le n!$.
It follows that 
$$\sigma_{k} = \red(a_{k}a_{k+1} \cdots a_{n!}) (n!-k+2) \cdots n = \red(a_{k}\cdots a_{n!} a_1 \cdots a_{k+n-n!-1})$$ 
for $n!-n+2 \le k \le n!$, which proves the theorem. 
\end{proof}

\subsection{Uniqueness of $\Pi_n$}\label{unique-sub}

The construction of $\Pi_n$ begins with the permutation $12\cdots  (n-1)$ and is followed by consecutive applications of the greedy algorithm choosing the smallest possible element to extend the already constructed sequence. A~natural question is what happens if we start with a different permutation of length $n-1$ and follow the greedy steps of our algorithm? Can we obtain other (non-equivalent) constructions of u-cycles for $n$-permutations? The following theorem answers this question.

\begin{thm}\label{uniqueness-thm} 
Beginning the construction of $\Pi_n$ {\rm (}and thus of $\Pi'_n${\rm )} with an $(n-1)$-permutation $\Pi'_{n,0}\neq 12\cdots (n-1)$ and following the steps of the greedy algorithm will {\em not} produce a $u$-cycle for $n$-permutations. 
\end{thm}

\begin{proof} 
Assume that $\sigma'_0 \neq 12\cdots(n-1)$.
As in the proof of Lemma~\ref{lem2}, the greedy algorithm terminates at~$k$ if $\sigma_k = 12\cdots n$. 
Now, the second part of the proof of Lemma~\ref{lem2} shows that $J_k < n$ when $\sigma'_k = 12\cdots(n-1)$. 
Therefore, the permutation $12\cdots n$ is not covered by $\Pi'_n$, hence $\Pi'_n$ is not a u-word and $\Pi_n$ is not a u-cycle. 
\end{proof}

\section{Properties of $\Pi_n$ and $\Pi'_n$}\label{prop-sec}

In this section, we shall give some properties of $\Pi_n$ and $\Pi'_n$.
All our results discussed are stated in terms of $\Pi'_n$, since the same results for $\Pi_n$ will trivially follow from the definition of $\Pi_n$.

Our main result is that all permutations with $1$ before $n$ occur in the first half of $\Pi'_n$, and those with $n$ before $1$ occur in the second half. 
More precisely, consider the partition of the symmetric group $S_n = S_{1,n} \cup S_{n,1}$, where
\begin{align*}
S_{1,n} & = \{\pi_1\pi_2\cdots \pi_n \in S_n:\ i < j\ \mbox{for}\ \pi_i = 1,\ \pi_j = n\}, \\[6pt]
S_{n,1} & = \{\pi_1\pi_2\cdots \pi_n \in S_n:\ i < j\ \mbox{for}\ \pi_i = n,\ \pi_j = 1\}.
\end{align*}

\begin{thm} \label{properties-thm}
We have $a_{\frac{n!}{2}+1} = 1$, 
\begin{gather*}
\sigma_{\frac{n!}{2}} = n12\cdots(n-1), \quad \sigma_{\frac{n!}{2}+1} = 134\cdots n2, \quad \sigma_{\frac{n!}{2}+2} = 23\cdots (n-1)1n, \\
\left\{\sigma_k:\, k \le \frac{n!}{2}\right\} = S_{n,1} \quad \mbox{and} \quad \left\{\sigma_k:\, k > \frac{n!}{2}\right\} = S_{1,n}.
\end{gather*}
\end{thm}

In particular, $\sigma_{\frac{n!}{2}} = n12\cdots(n-1)$ is the last occurrence of an $n$-permutation starting with~$n$, $\sigma_{\frac{n!}{2}+1} = 134\cdots n2$ is the first occurrence of an $n$-permutation starting with~$1$, and $\sigma_{\frac{n!}{2}+2} = 23\cdots (n-1)1n$ is the first occurrence of an $n$-permutation ending with~$n$.

For the proof of Theorem~\ref{properties-thm}, we use the following lemmas. 

\begin{lem} \label{l:am}
Let $m \ge 1$ be minimal such that $\sigma_m$ starts with $1$.
Then $a_m = 1$, and $\sigma_k$ does not end with $n$ for all $k \le m$. 
\end{lem}

\begin{proof}
The fact that $a_m = 1$ is a direct consequence of the greedy algorithm. 
Suppose that $\sigma_k$ ends with $n$ for some $k \le m$. 
Then we have $\sigma'_j = \sigma'_{k-1}$ for $n$ different indices $j < k$. 
Since $\sigma_k \ne 12\cdots n$, we have that $j \ge 1$ for all these~$j$.  Hence, there are $n$ different permutations $\sigma_j$ ending with~$\sigma'_{k-1}$. 
Therefore, one of these permutations starts with $1$, contradicting the definition of~$m$.
\end{proof}

\begin{lem} \label{l:Sn1}
Let $m$ be as in Lemma~\ref{l:am}.
We have $\sigma_k \in S_{n,1}$ for all $k < m$, $\sigma_{m-1} = n12\cdots(n-1)$ and $\sigma_m = 134\cdots n2$.
\end{lem}

\begin{proof}
We use induction on~$k$. 
We have $\sigma_1 = 23\cdots n1 \in S_{n,1}$. 
Assume that $\sigma_j \in S_{n,1}$ for all $j \le k$, where $1 \le k < m$.
If $\sigma'_k \in S_{n-1,1}$, then we have $\sigma_{k+1} \in S_{n,1}$ because $\sigma_{k+1}$ does not end with~$n$ by Lemma~\ref{l:am}.
Now, assume that $\sigma'_k \in S_{1,n-1}$.
Since $\sigma_k \in S_{n,1}$, $\sigma_k$ starts with~$n$.
Moreover, we cannot have $\sigma'_j = \sigma'_k$ for $1 \le j < k$,
since otherwise it implies that $\sigma'_j \in S_{1,n-1}$.  Together with  $\sigma_j \in S_{n,1}$, it follows that
 $\sigma_j$ would also have to start with $n$,  contradicting that $\sigma_j \ne \sigma_k$. 
If $\sigma'_k \ne \sigma'_0$, then this implies that $\sigma_{k+1}$ ends with~$1$ and hence $\sigma_{k+1} \in S_{n,1}$. 
Finally, if $\sigma'_k = \sigma'_0$, then $\sigma_{k+1}$ is an $i$-th extension of $12\cdots(n-1)$ with $i \ge 2$, thus $\sigma_{k+1}$ starts with~$1$, which is possible only for $k=m-1$. 
Therefore, we have $\sigma_k \in S_{n,1}$ for all $k < m$.

We have $\sigma_m \in S_{1,n}$ because it starts with~$1$, hence we obtain from the previous paragraph that $\sigma'_{m-1} = \sigma'_0$, $\sigma_{m-1} = n12\cdots(n-1)$ and $\sigma_m = 134\cdots n2$.
\end{proof}

\begin{lem} \label{l:S1n}
Let $m$ be as in Lemma~\ref{l:am}.
We have $\sigma_k \in S_{1,n}$ for all $k \ge m$, with $\sigma_{m+1} = 23\cdots (n-1)1n$.
\end{lem}

\begin{proof}
Suppose that $\sigma_k \in S_{n,1}$ for some $k \ge m$.
Then, we will show that, for some $i \ge m$, $\sigma_i$ starts with $n$ and ends with $n-1$.
We obtain that, for some $i \ge m$, $\sigma_i$ ends with $n1$, and finally this yields that $n12\cdots (n-1)$ occurs after~$m$, contradicting that $\sigma_{m-1} = n12\cdots(n-1)$.

First, note that $\sigma'_k = \sigma'_0$ and $\sigma_k \in S_{n,1}$ are possible only for $k = m-1$.
Therefore, $\sigma_k \in S_{n,1}$ for $k \ge m$ implies that $\sigma'_j = \sigma'_k$ for at most $n$ different indices $j \le k$.
Moreover, since $\sigma_j$ cannot start with~$1$ if $j < m$ or $j = k$, and since $k \ge m$, there are at most $n-2$ different indices $j < m$ such that $\sigma'_j = \sigma'_k$.
Hence the $(n-1)$-st extension of $\sigma'_k$ occurs after~$m$.

If $\sigma_k \in S_{n,1}$ does not start with~$n$, then its $(n-1)$-st extension is also in~$S_{n,1}$, with the position of $n$ shifted to the left; we obtain iteratively that there is some $i \ge m$ such that $\sigma_i$ starts with $n$ and ends with $n-1$. 
On the other hand, if $\sigma_k$ starts with~$n$, then the $(n-1)$-st extension of $\sigma'_{k-1}$ starts with $n$ and occurs at or after the position~$k$. 

Therefore, we can assume that, for some $k \ge m$, $\sigma_k$ starts with $n$ and ends with $n-1$.
Now, we cannot have $\sigma'_j = \sigma'_k$ for some $1 \le j < m$, because $\sigma_j \ne \sigma_k$ would imply that $\sigma_j$ ends with $n$, which is not possible for $j < m$.
Since $\sigma'_k \ne \sigma'_0$, we obtain that $\sigma'_k$ does not occur before~$m$, hence the first extension of $\sigma'_k$, which ends with $n1$, occurs after~$m$. 

Finally, assume that $\sigma_k$ ends with $n1$ for some $k \ge m$.
Then, iterating $(n-1)$-st extensions gives that $n12\cdots(n-1)$ occurs after~$m$, contradicting that $\sigma_{m-1} = n12\cdots(n-1)$. 
Hence, we cannot have $\sigma_k \in S_{n,1}$ for $k \ge m$.

Since $\sigma'_m = 23\cdots (n-1)1$ by Lemma~\ref{l:Sn1} and $\sigma_{m+1} \in S_{1,n}$, we have $\sigma_{m+1} = 23\cdots (n-1)1n$.
\end{proof}

\begin{proof}[Proof of Theorem~\ref{properties-thm}]
Since $|S_{1,n}| = |S_{n,1}| = n!/2$, the theorem is a direct consequence of Lemmas~\ref{l:Sn1} and~\ref{l:S1n}, with $m =  n!/2+1$.
\end{proof}

Note that, since $a_{\frac{n!}{2}+1} = 1$, we have, for all $k \le  n!/2-n+2$, some $0 \le j \le n-2$ such that $\sigma_{k+j}$ ends with~$1$.
This also follows from the fact that $\sigma_k \in S_{n,1}$ for all $k \le  n!/2$.
Similarly, since $\sigma_k \in S_{1,n}$ for all $k >  n!/2$, we have, for all $k >  n!/2$, some $0 \le j \le n-2$ such that $\sigma_{k+j}$ ends with~$n$.
Therefore, $a_{n!+n-1}$ is the largest element to the right of $a_{\frac{n!}{2}+1} = 1$.

\section{Directions of further research}\label{further-res-sec}

\subsection{Reducing the alphabets of $\Pi'_n$ and $\Pi_n$}\label{reducing-sec}

Our main goal in this paper was to come up with a {\em simple} construction of a u-cycle for permutations, {\em not} a construction requiring the {\em minimum} number of letters, which was a concern in, say~\cite{CDG1992,HRW2012,Johnson2009}. However, it is an interesting and challenging question to ask what is the minimum number of distinct letters that is required for our greedy construction.  These considerations will be similar in nature to dealing with partially ordered sets in \cite{CDG1992}.

For example, constructing $\Pi'_3$ in a greedy way with respect to the number of letters used, we end up with a word over a 6-letter alphabet, not over an 8-letter alphabet: 
\begin{align*}
 12  & \rightarrow
 231 \rightarrow
 3421 \rightarrow
  45312 \rightarrow
 564132 \rightarrow  
 5641324 \rightarrow
 56413245.
\end{align*}
Such an optimization is irrelevant for the construction of $\Pi_3$, but it would be relevant for the construction of $\Pi_n$ for $n\geq 4$. So, the idea is that in the original construction of $\Pi'_n$ some of the letters can be used more than once. This defines a partial order on the letters of the originally constructed $\Pi'_n$ indicating which letters have a choice to be used more than once. More precisely, the poset is defined by $a_j \prec a_k$ if $a_j < a_k$ and $|k-j|<n$, where $\prec$ denotes the poset relation. Then the minimal number of letters is the length of this partial order. Such partially ordered sets, in the cases of $n=3,4$, are shown in Figure~\ref{posets-n=3-4}, and we can see that $\Pi'_3$ (resp., $\Pi'_4$) actually requires 6 (resp., 13) distinct letters.  It would be interesting to study properties of such partially ordered sets. For example, from the case of $n=4$, we see that such partially ordered sets are not necessarily graded. 

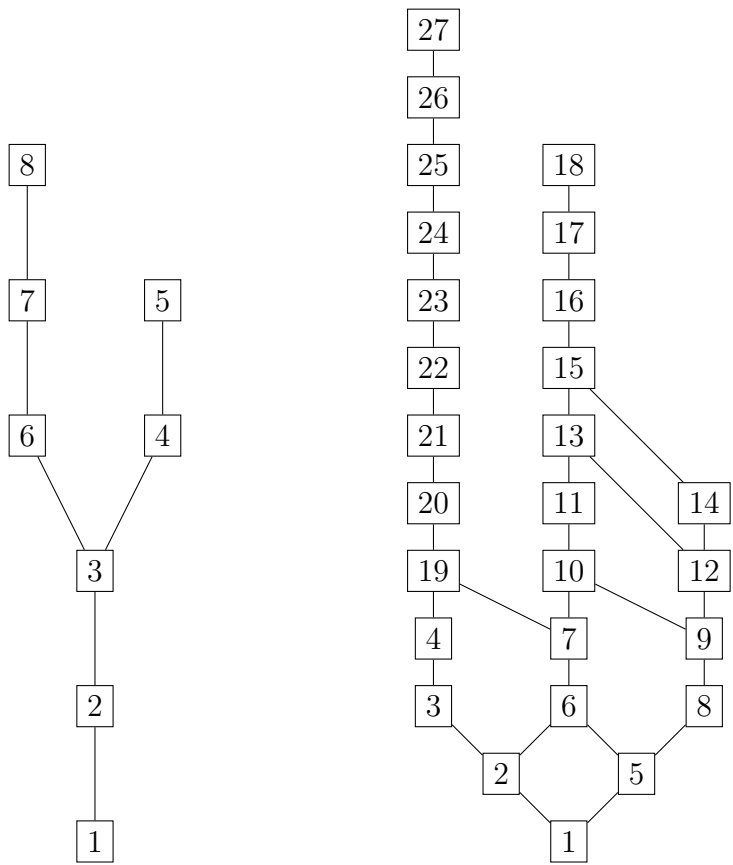
\begin{figure}[ht]
\centering
\begin{tikzpicture}[scale=.9]
\node (1) at (0,0) [draw] {$1$};
\node (2) at (0,2) [draw] {$2$};
\node (3) at (0,4) [draw] {$3$};
\node (4) at (1,6) [draw] {$4$};
\node (5) at (1,8) [draw] {$5$};
\node (6) at (-1,6) [draw] {$6$};
\node (7) at (-1,8) [draw] {$7$};
\node (8) at (-1,10) [draw] {$8$};
\draw (1)--(2)--(3)--(4)--(5) (3)--(6)--(7)--(8);

\begin{scope}[shift={(7,0)}]
\node (1) at (0,0) [draw] {$1$};
\node (2) at (-1,1) [draw] {$2$};
\node (3) at (-2,2) [draw] {$3$};
\node (4) at (-2,3) [draw] {$4$};
\node (5) at (1,1) [draw] {$5$};
\node (6) at (0,2) [draw] {$6$};
\node (7) at (0,3) [draw] {$7$};
\node (8) at (2,2) [draw] {$8$};
\node (9) at (2,3) [draw] {$9$};
\node (10) at (0,4) [draw] {$10$};
\node (11) at (0,5) [draw] {$11$};
\node (12) at (2,4) [draw] {$12$};
\node (13) at (0,6) [draw] {$13$};
\node (14) at (2,5) [draw] {$14$};
\node (15) at (0,7) [draw] {$15$};
\node (16) at (0,8) [draw] {$16$};
\node (17) at (0,9) [draw] {$17$};
\node (18) at (0,10) [draw] {$18$};
\node (19) at (-2,4) [draw] {$19$};
\node (20) at (-2,5) [draw] {$20$};
\node (21) at (-2,6) [draw] {$21$};
\node (22) at (-2,7) [draw] {$22$};
\node (23) at (-2,8) [draw] {$23$};
\node (24) at (-2,9) [draw] {$24$};
\node (25) at (-2,10) [draw] {$25$};
\node (26) at (-2,11) [draw] {$26$};
\node (27) at (-2,12) [draw] {$27$};
\draw (1)--(2)--(3)--(4)--(19)--(20)--(21)--(22)--(23)--(24)--(25)--(26)--(27) (1)--(5)--(6)--(7)--(19) (2)--(6) (5)--(8)--(9)--(10)--(11)--(13)--(15)--(16)--(17)--(18) (7)--(10) (9)--(12)--(13) (12)--(14)--(15);
\end{scope}
\end{tikzpicture}

 	\caption{Partially ordered sets for $\Pi'_3$ and $\Pi'_4$}\label{posets-n=3-4}
 \end{figure}

We note that our greedy algorithm will never result in the optimum size of the alphabet, which is $n+1$ for $n$-permutations, because of the first $n$ steps of the algorithm when we are forced to use new letters, so the total number of letters will be $\geq 2n-2>n+1$ for $n\geq 4$. However, to find the exact optimal size for our construction is an interesting problem requiring understanding the partially ordered sets in question. 

If constructing a universal cycle for $n$-permutations over an $(n+1)$-letter alphabet is someone's main goal, it would be interesting to see if there is a (natural) modification of our algorithm that would help to achieve this goal (thus reproving the result of Johnson~\cite{Johnson2009}). Such a modification would be in describing (in a ``nice'' way) a sequence of extensions instead of applying the minimal possible extension every time. For example, we might be interested in alternation of applications of minimal and maximal extensions, or applying some sort of closest to average extension, etc. Of course, many, if not (almost) all such modifications will not result in a universal cycle. This effectively opens up the question on classification of sequences of extensions that result in universal cycles (leaving aside the question on the number of letters used by a particular universal cycle obtained this way).

\subsection{Shortened universal words/cycles for permutations}\label{sec-shortened}

It would be interesting to extend our approach to generate universal cycles for permutations in the context of shortened universal words/cycles for permutations that were introduced in \cite{KPV2017}. 

To illustrate the idea of one of the two ways suggested in  \cite{KPV2017} to shorten universal cycles for permutations, consider the word $112$, which is claimed  to be a universal cycle for all permutations of length 3, thus shortening a ``classical'' universal cycle for these permutations, say, $145243$. Indeed, we can treat equal elements as {\em incomparable elements}, while the relative order of these incomparable elements to the other elements must be respected. Thus, $112$ encodes all permutations whose last element is the largest one, namely, $123$ and $213$; starting at the second position (and reading the word cyclically), we obtain the word $121$ encoding the permutations $132$ and $231$, and finally, starting at the third position, we (cyclically) read the word $211$ encoding the permutations $312$ and $321$.  For another example, the word $1232$ is also a universal cycle for permutations of length $3$. 

The main goal of \cite{KPV2017} is to study compression possibilities for (classical) universal cycles for permutations. In particular, \cite{KPV2017} shows that such universal cycles exist of lengths $n!-kn$ for $k=0,1,\ldots,(n-1)!$. The approach in \cite{KPV2017} to obtain these results is graph theoretical, and it actually does not provide any explicit constructions of the objects in the general case. Thus, (a~modification of) our approach to generate universal cycles could potentially be useful in generating shortened universal words/cycles, and this is an interesting direction of further research. 
An idea here would be to introduce a rule, or rules, overriding the rules of the algorithm, while still beginning with the increasing sequence of length $(n-1)$ for $n$-permutations. For example, we could require to use the minimum letter equal to some letter among the $(n-1)$ preceding letters, if possible, while following our original algorithm otherwise. So that the steps of such a modified algorithm in the case of $n=3$ could be
$$12  \rightarrow
 121 \rightarrow
 1211 \rightarrow
  12112$$  
covering all 6 permutations non-cyclicly, or for the cyclic version, we can stop at 121 (132 and 231 are covered by 121,  321 and 312 by 211, and 123 and 213 by 112). However, correctness of such an algorithm, or its variations, must be addressed, which is outside of the scope in this paper.   

\section*{Acknowledgments}
The first author was supported by the Fundamental Research Funds for the Central Universities (31020170QD101). 
The second author is grateful to the administration of the Center for Combinatorics at Nankai University for their hospitality during the author's stay in May 2017. 
The last author was partially supported by the National Science Foundation of China (Nos. 11626172, 11701424).

\end{document}